\documentclass[12pt]{amsart}      
\usepackage{amssymb,enumerate}
\usepackage{amsmath}
\usepackage{setspace}
\usepackage{tcolorbox}

\usepackage{float,enumitem}
\usepackage{mathscinet}

\usepackage{amsmath,amssymb,amsfonts,amsthm}

\newtheorem{thm}{Theorem}
\newtheorem{lemma}{Lemma} 
\newtheorem{prop}{Proposition} 
\newtheorem{corollary}{Corollary} 

\theoremstyle{definition} 
\newtheorem{remark}{Remark} 
\newtheorem{definition}{Definition} 
\newtheorem{example}{Example}

\newcommand{\U}{\mathcal{U}}

\newcommand{\K}{\mathcal{K}}

\newcommand{\card}{\operatorname{card}}

\newcommand{\dH}{d^{\mathrm{H}}}

\newcommand\R{\mathbb{R}} 

\newcommand\N{\mathbb{N}}

\newcommand{\be}{\begin{equation}}
\newcommand{\ee}{\end{equation}}

\newcommand{\M}{{\mathcal{M}}}
\renewcommand{\P}{{\mathcal{P}}}
\newcommand{\NN}{{\mathrm{N}}}

\newcommand{\KK}{\mathcal K}

\newcommand{\vp}{\varphi}
\newcommand{\B}{{\mathcal{B}}}

\newcommand{\A}{{\mathcal{A}}}

\newcommand{\ep}{{\varepsilon}}

\newcommand{\spa}{\operatorname{span}}
\newcommand{\sepp}{\operatorname{sep}}

\newcommand{\orbn}{{\operatorname{Orb}_n}}
\newcommand\orb[1]{{\operatorname{Orb}_{#1}}}
\newcommand{\KT}{\operatorname{KT}}

\newcommand{\AK}{\operatorname{AKEK}}

\newcommand{\hKT}{h_{\KT}}

\newcommand{\NZ}{\operatorname{NZ}}

\newcommand{\CM}{\operatorname{CMM}}
\newcommand{\hCM}{h_{\CM}}

\usepackage{mathscinet}

\begin{document}

\begin{abstract}
The main aim of this note is to point out by means of counter-examples that some arguments of the proofs of two theorems about a \emph{half variational principle} for multivalued maps, formulated recently by Vivas and Sirvent [Metric entropy for set-valued maps, Discrete Contin.
Dyn. Syst. Ser. B, 27 (2022), pp. 6589–6604], are false and that our corrected versions require rather restrictive additional assumptions. Nevertheless, we will be able to establish the full variational principle for a special subclass of multivalued lower semicontinuous maps with convex compact values on a compact subset of a Banach space.
\end{abstract}


\thanks{The authors were supported by the Grant IGA\_PrF\_2024\_006 ``Mathematical Models'' of the Internal Grant Agency of Palack\'y University in Olomouc.}


\title[Note on the metric entropy for multivalued maps]{Note on the metric entropy for multivalued maps}

\author[J. Andres]{Jan Andres}
\address{Department of Mathematical Analysis and Applications of Mathematics, Faculty of Science, Palack\'y University, 17. listopadu 12, 771 46 Olomouc, Czech Republic}
\author[P. Ludv\'\i k]{Pavel Ludv\'\i k}
\email{pavel.ludvik@upol.cz}

\maketitle

\section{Introduction}
The story of entropy for single-valued maps has been described in detail and seems to be well understood (see e.g. \cite{Ba1,Ba, Ba2,Do, FN, Ka}, and the references therein). For multivalued maps, the situation is however more delicate. Although for topological entropy some papers already exist (see e.g. \cite{AK,An,AL1,AL2,AL3,CMM1,CP,EK,KT,RT,WZZ}), the results for metric entropy are still quite rare (see e.g. \cite{CMM2,LLZ,VS}).

The relationship between metric and topological entropies for single-valued continuous maps on a compact metric space can be effectively expressed by means of a \emph{variational principle}, which was formulated for the first time in 1970 by Dinaburg \cite{Di} on a finite-dimensional space and, independently, in 1971 by Goodman \cite{Go3} on an arbitrary compact metric space (cf. also \cite{Mi}). More concretely, they compared the topological entropy $h(f)$ of $f:X\to X$ with the measure-theoretic entropy $h_{\mu}(f)$ of $f$, for $f$-invariant regular probability measures $\mu$ on the Borel sets of $X$ and proved that the supremum of the metric entropies over all Borel probability invariant measures is equal to the topological entropy. Let us note that already in 1969 Goodwyn \cite{Go1} has shown that topological entropy bounds measure-theoretic (i.e. metric) entropy from above and in 1972 he extended in \cite{Go2} his result to compact Hausdorff spaces.

In a multivalued case, as far as we know, there is only an analogous inequality to the one of Goodwyn in \cite[Theorem 3.2]{VS}, where the topological entropy is understood in the sense of \cite{KT}. The reverse inequality is examined under certain special restrictions in \cite[Theorem 3.7]{VS} and in \cite[Theorem 1.4]{CMM2}, where the topological entropy is considered this time in the sense of \cite{CMM1}. Although both definitions of topological entropy for multivalued maps are consistent with the one for the single-valued case, their properties significantly differ each to other.

Unfortunately, the proofs of the mentioned results in \cite[Theorems 3.2 and 3.7]{VS} are not correct. It stimulated our reaction in this note, which is organized as follows.

At first, we will recall the related definitions of metric and topological entropy for multivalued maps and the ``half variational principles'', as it was presented in \cite{VS}. Since the standard Adler-Konheim-McAndrew type definitions (cf. \cite{AKM}) and the equivalent Bowen-Dinaburg type definitions (cf. \cite{Bo,Di}) for single-valued continuous maps on compact metric spaces can be regarded as particular cases of their multivalued extensions, we will omit them here. Then we will show  the counter-examples demonstrating the non-validity of the arguments applied to principles in \cite{VS}. Finally, we will present with comments their particular corrected versions and especially formulated the full variational principle for special lower semicontinuous maps with convex, compact values in a compact subset of a Banach space.

\section{Preliminaries}\label{sec:2}
Le $(X,d)$ be a compact metric space and $(X,\B,\mu)$ be a measure space, where $\B$ is the Borel $\sigma$-algebra of $X$ generated by open subsets of $X$ and $\mu$ stands for a $\sigma$-additive measure defined in $\B$. If $\mu(X)=1$, then it is called a \emph{probability measure}.

By a \emph{multivalued map} $\vp:X\multimap Y$, we mean the one with nonempty values, i.e. $\vp:X\to 2^{Y}\setminus\{\emptyset\}$. If $\vp$ has still closed values, then we write $\vp:X\to\K(Y)$, where $\K(Y):=\{K\subset Y: \mbox{ $K$ is compact}\}$.

The regularity of semicontinuous multivalued maps can be defined by means of the preimages of $\vp\colon X\multimap Y$, where
\begin{equation*}
\vp^{-1}_-(B):= \{x\in X:\vp(x)\subset B\} \mbox{ (``small'' preimage of $\vp$ at $B\subset Y$),}
\end{equation*}
resp.
\begin{equation*}
\vp^{-1}_+(B):= \{x\in X:\vp(x)\cap B\neq\emptyset\} \mbox{ (``large'' preimage of $\vp$ at $B\subset Y$).}
\end{equation*}

\begin{definition}\label{d:1}
\begin{enumerate}
	\item[(i)] $\vp:X\multimap Y$ is said to be \emph{upper semicontinuous} (u.s.c.) if $\vp^{-1}_-(B)$ is open for every open $B\subset Y$, resp. $\vp^{-1}_+(B)$ is closed for every closed $B\subset Y$;
	\item[(ii)] $\vp:X\multimap Y$ is said to be \emph{lower semicontinuous} (l.s.c.) if $\vp^{-1}_+(B)$ is open for every open $B\subset Y$, resp. $\vp^{-1}_-(B)$ is closed for every closed $B\subset Y$;
	\item[(iii)] $\vp:X\multimap Y$ is said to be \emph{continuous} if it is both u.s.c. and l.s.c.
\end{enumerate}
\end{definition}

Obviously, if $\vp:X\to Y$ is single-valued u.s.c. or l.s.c., then it is continuous. Moreover, if $\vp:X\to\KK(Y)$ is u.s.c. and $K\in\KK(X)$, then $\bigcup_{x\in K} \vp(x)\in\KK(Y)$ (see e.g. \cite[Proposition I.3.20]{AG}, \cite[Corollary 1.2.20]{HP})

The celebrated \emph{Michael selection theorem} \cite{Ma} (cf. also \cite[Theorem 2.1]{ATDZ} and \cite[Theorem 1.4.6]{HP}) will be formulated here in the form of proposition.

\begin{prop}[cf. \cite{Ma}]\label{p:1}
Let $X$ be a compact metric space, $Y$ be a Banach space and $\vp:X\multimap Y$ be an l.s.c. multivalued map with closed convex values. Then there exists a single-valued continuous selection $f\subset\vp$ of $\vp$, i.e. $f(x)\in\vp(x)$, for every $x\in X$.
\end{prop}

For more details concerning multivalued analysis, see e.g. \cite{AF,HP}.

\begin{definition}\label{d:2}
Let $(X,\B,\mu)$ be a measure space. Then the Borel probability measure $\mu$ is \emph{invariant} under $\vp:X\to\K(X)$ (i.e. $\vp$-invariant) if the inequality
\[
\mu(\vp^{-1}_+(A)) \geq \mu(A)
\]
holds for every $A\in\B$.
\end{definition}

It is well known that, under the above assumptions, an invariant probability measure $\mu$ always exists (for u.s.c. maps $\vp$ on compact sets $X$, see e.g. \cite[Theorem 8.9.4]{AF}, and for l.s.c. maps $\vp$ with convex values on a compact subset $X$ of a Banach space, see \cite[Theorem 2.7]{TWY}). Let us denote by $\M_1(\vp)$ the set of all $\vp$-invariant measures.

By a (measurable) partition $\P\subset\B$ of $X$, we mean a countable family of pairwise disjoint measurable sets whose union has a full measure. The \emph{entropy of a measurable partition} $\P$ is given by
\[
H_\mu(\P):= -\sum_{P\in\P}\mu(P)\log \mu(P) = -\sum_{P\in\P} \Phi(\mu(P)),
\]
where $\Phi:[0,\infty)\to\R$ is defined as
\[
\Phi(x) =
\begin{cases}
&x\cdot\log x,\mbox{ for $x\in(0,\infty)$,}\\
&0,\mbox{ for $x=0$.}
\end{cases}
\]
For the elementary properties of $\Phi$, see e.g. \cite[Theorem 4.2]{Wa}.

Furthermore, the \emph{metric entropy $h_\mu(\vp,\P)$ of $\vp$} with respect to $\mu$ of a totally ordered finite partition $\P=\{P_1,\ldots,P_n\}$ of $X$ is given by
\[
h_\mu(\vp,\P) := \limsup_{n\to\infty} \frac1{n} H_\mu(\P'_n),
\]
where the elements of $\P'_n$ take the form $P_0\cap P_1\cap\ldots\cap P_{n-1}$, $P_k\in\tilde{\P}_k$, for $k=0,\ldots,n-1$, and $\tilde{P}_k$ is the refinement of $\P_k=\vp^{-k}_+(\P)$, $k\geq 1$, whose elements are defined as
\[
\tilde{P}_{1,k} = \vp^{-k}_+(P_1), \tilde{P}_{j,k} = \vp^{-k}_+(P_j)\setminus\bigcup_{i<j} \vp^{-k}_+(P_i), j=2,\ldots,n.
\]
Thus, we consider the following sequence of partitions of $X$:
\[
\tilde{\P}_0=\P, \P'_k = \tilde{\P}_0\vee\tilde{\P}_1\vee\ldots\vee\tilde{P}_{k-1} = \bigvee_{i=0}^{k-1} \tilde{P}_i, k\geq 1,
\]
where $\bigvee_{i=0}^{k-1} \tilde{P}_i = \{\bigcap_{i=0}^{k-1}\tilde{P}_{j,i}:\tilde{P}_{j,i}\in\tilde{P}_i, i=0,\ldots,k-1,\mbox{ and } \bigcap_{i=0}^{k-1}\tilde{P}_{j,i}\neq\emptyset, j=1,\ldots,n\}$, $k\geq 1$.

\begin{definition}[cf. {\cite[Definition 2.1]{VS}}]\label{d:3}
The \emph{metric entropy} $h_\mu(\vp)$ of $\vp:X\to\K(X)$ is given by
\[
h_\mu(\vp) := \sup_{\P} h_\mu(\vp,\P) = \sup_{\P} \limsup_{n\to\infty} \frac1{n} H_\mu(\bigvee_{k=0}^{n-1} \tilde{P}_k),
\]
where the supremum is taken over all finite totally ordered partitions $\P$ of $X$.
\end{definition}

\begin{remark}
In a single-valued case, Definition~\ref{d:3} coincides with the standard definitions of metric entropy considered in \cite{Di,Go1,Go3,Mi,Wa}.
\end{remark}

The following useful lemma is a variant of the well known result (see e.g. \cite[Corollary 4.2.1]{Wa}), suited for our purposes.

\begin{lemma}\label{l:1}
Let $(X,\mu)$ be a measurable space and $\P=\{A_1,\ldots,A_k\}$ be a measurable partition of $X$. Then
\[
H_{\mu}(\P) \leq \log \NZ(\P),
\]
where $\NZ(\P)$ is the number of sets in $\P$ with a non-zero measure.
\end{lemma}
\begin{proof}
From the definition of the entropy $H_{\mu}(\P)$ of the measurable partition $\P$, we have
\[
H_{\mu}(\P) = -\sum_{i=1}^{k} \Phi(\mu(A_i)).
\]

If $\NZ(\P)=0$, then there is nothing to prove. Assume $\NZ(\P)>0$ and
let us define the set of indices $i\in\{1,\ldots,k\}$ with $\mu(A_i)\neq 0$ as $I$. Then $\card I = \NZ(\P)$. Since $\Phi(0)=0$ and $\Phi$ is strictly convex on $[0,\infty)$, we can write
\begin{align*}
H_{\mu}(\P) &= -\sum_{i\in I} \Phi(\mu(A_i)) = 
-\NZ(\P) \sum_{i\in I} \frac1{\NZ(\P)}\Phi(\mu(A_i)) \\
&\leq -\NZ(\P) \Phi\left(\frac1{\NZ(\P)} \sum_{i\in I} \mu(A_i)\right) = -\NZ(\P) \Phi(\NZ(\P)^{-1}) \\
&= -\NZ(\P)\cdot\NZ(\P)^{-1}\cdot \log(\NZ(\P)^{-1}) = \log \NZ(\P). 
\end{align*}
\end{proof}

It will be convenient to recall several definitions of topological entropy for multivalued maps given in \cite{An,AL1,AL2,AL3,CMM1,KT}. We start with the Adler-Konheim-McAndrew type (cf. \cite{AKM}) definitions in \cite{An,AL1}.

Hence, consider the l.s.c. multivalued map $\vp:X\multimap X$ on a compact metric space $X$. For open covers $\A_j$, $j=0,\ldots,n-1$, of $X$, we define again
\[
\bigvee_{j=0}^{n-1} \A_j := \{ \bigcap_{j=0}^{n-1} A_j: A_j\in\A_j, j=0,\ldots,n-1,\mbox{ and } \bigcap_{j=0}^{n-1}A_j\neq\emptyset\}.
\]

Since for open covers $\A$ of $X$ we have (cf. \cite{An,AL1})
\[
\vp^{-j}_+(\A) = \{\vp^{-j}_+(A):A\in\A\}, j\geq0,
\]
we can put
\[
		\A^n_{\vp_+} := \bigvee_{j=0}^{n-1} \vp_{+}^{-j}(\A) = \bigvee_{j=0}^{n-1} \{\vp_{+}^{-j}(A): A\in\A\}.
\]

Taking the growth rate of the number of elements in $\A^n_{\vp_+}$ by (cf. \cite[Lemma 3]{AL1})
\[
h_+(\vp,\A):= \lim_{n\to\infty} \frac1{n} \log(\NN(\A^n_{\vp_+})),
\]
where $\NN(\A^n_{\vp_+})$ is the minimal cardinality of a subcover $\B\subset\A^n_{\vp_+}$, the topological entropy $h_+(\vp)$ of $\vp$ can be defined in the following way.

\begin{definition}[cf. {\cite[Definition 6]{AL1}}]\label{d:4}
Let $\vp:X\multimap X$ be an l.s.c. multivalued map on a compact metric space $X$. The \emph{topological entropy $h_+(\vp)$ of $\vp$} is given by
\[
h_+(\vp):= \sup \{h_+(\vp,\A): \mbox{ $\A$ is an open cover of $X$}\}.
\]
\end{definition}

\begin{remark}
Let us note that, unlike to \cite[Definition 6]{AL1}, we have defined here the topological entropy $h_+$ for l.s.c. maps $\vp:X\multimap X$ with not necessarily closed values, which will be convenient for the following lemma and its application in the proof of Theorem~\ref{t:2} below. The closed values of l.s.c. maps need not be preserved for their compositions and, in particular their iterates (see e.g. \cite{ATDZ,HP}).
\end{remark}

\begin{lemma}\label{l:iter}
Let $X$ be a compact metric space and $\vp:X\multimap X$ be an l.s.c. map. Then
\begin{align}\label{eq:iter}
h_{+}(\vp^k) \leq k\cdot h_{+}(\vp),
\end{align}
for any $k\in\N$.
\end{lemma}
\begin{proof}
Let $k\in\N$ and $\A$ be an open cover of $X$. Then
\[
\A\vee (\vp^k)^{-1}_{+}(\A) \vee \ldots \vee (\vp^{k})^{-n+1}_{+}(\A) \prec \A\vee \vp^{-1}_{+}(\A) \vee \ldots \vee \vp^{-nk+1}_{+}(\A),
\]
where the symbol ``$\prec$'' denotes that the right-hand side is finer than the left-hand side.

Therefore (see also \cite[Lemma 3]{AL2}),
\begin{align*}
h_{+}(\vp) \geq h_{+}(\vp,\A) &= \lim_{n\to\infty} \frac1{nk} \log \NN(\A\vee \vp^{-1}_{+}(\A) \vee \ldots \vee \vp^{-nk+1}_{+}(\A))\\
&\geq \frac1{k} \lim_{n\to\infty} \frac1{n} \log \NN(\A\vee (\vp^k)^{-1}_{+}(\A) \vee \ldots \vee (\vp^{k})^{-n+1}_{+}(\A))\\
&= \frac1{k} h_{+}(\vp^k,\A).
\end{align*}

Passing to the supremum over all open covers $\A$, we obtain \eqref{eq:iter}.

\end{proof}

\begin{lemma}[cf. {\cite[Proposition 7]{AL1}}]\label{l:2}
Let $\vp:X\to\K(X)$ be an l.s.c. multivalued map on a compact metric space $X$ and assume that $\vp$ possesses a single-valued continuous selection $f\subset\vp$, i.e. $f(x)\in\vp(x)$, for every $x\in X$. Then the inequality 
\begin{align}\label{eq:A1}
h_+(\vp)\leq h(f):=h_+(f)
\end{align}
holds for the topological entropy $h_+$ in the sense of Definition~\ref{d:4}.
\end{lemma}

Now, we will recall the Bowen-Dinaburg type definitions from \cite{KT} and \cite{CMM1}.

For multivalued maps $\vp:X\multimap X$, we will consider the sets of \emph{n-orbits} of $\vp$ as follows:
\begin{align*}
\orb{1}(\vp)&:=X,\\
\orbn(\vp) &:= \{(x_0,\ldots, x_{n-1})\in X^{n}: x_{i+1} \in \vp(x_i), i=0,\ldots, n-2\},\, n\geq 2.
\end{align*}

\begin{definition}\label{d:5}
Let $X=(X,d)$ be a metric space and $\ep>0$. A set $S\subset X$ is called \emph{$\ep$-separated} if $d(x,y)>\ep$ holds for every pair of distinct points $x,y\in S$. A set $R\subset X$ is called \emph{$\ep$-spanning} in $Y\subset X$ if, for every $y\in Y$, there is $x\in R$ such that $d(x,y)\leq\ep$.
\end{definition}

\begin{definition}\label{d:6}
Let $\vp:X\multimap X$ be a multivalued map on a metric space $X=(X,d)$ and $\ep>0$. Denoting by $d_n$ the metric on $X^n$ defined as
\[
d_n((x_0,\ldots,x_{n-1}),(y_0,\ldots,y_{n-1})):=\max\{d(x_i,y_i):i=0,\ldots,n-1\},
\]
we call $S\subset\orbn(\vp)$ an \emph{$(\ep,n)_{\KT}$-separated set} for $\vp$ if it is an $\ep$-separated subset of the metric space $(X^n,d_n)$. We call $R\subset\orbn(\vp)$ an \emph{$(\ep,n)_{\KT}$-spanning set} for $\vp$ if it is an $\ep$-spanning subset in $\orbn(\vp)\subset X^n$.
\end{definition}

\begin{definition}[cf. \cite{KT}]\label{d:7}
Let $X=(X,d)$ be a compact metric space and $\vp:X\multimap X$ be a multivalued map. Denoting by $s_{\KT}(\vp,\ep,n)$ the largest cardinality of an $(\ep,n)$-separated set for $\vp$, we take

\begin{equation*}
\hKT^{\sepp}(\vp,\ep) := \limsup_{n\to\infty} \frac1{n}\log s_{\KT}(\vp,\ep,n).
\end{equation*}
The \emph{topological entropy} $\hKT^{\sepp}(\vp)$ of $\vp$ is defined to be
\begin{equation*}
\hKT^{\sepp}(\vp):= \sup_{\ep>0} \hKT^{\sepp}(\vp,\ep).
\end{equation*}

Denoting by $r_{\KT}(\vp,\ep,n)$ the smallest cardinality of an $(\ep,n)_{\KT}$-spanning set for $\vp$, we take
\begin{equation*}
\hKT^{\spa}(\vp,\ep) := \limsup_{n\to\infty} \frac1{n}\log r_{\KT}(\vp,\ep,n).
\end{equation*}
The \emph{topological entropy} $\hKT^{\spa}(\vp)$ of $\vp$ is defined to be
\begin{equation*}
\hKT^{\spa}(\vp):= \sup_{\ep>0} \hKT^{\spa}(\vp_{0,\infty},\ep).
\end{equation*}
\end{definition}

\begin{remark}\label{r:2}
Although the multivalued maps under consideration in \cite{KT} were defined to be u.s.c. with closed values, we have shown in \cite{AL3} that the maps can be quite arbitrary and since $\hKT^{\sepp}(\vp)=\hKT^{\spa}(\vp)$, we can put
\[
\hKT(\vp):= \hKT^{\sepp}(\vp)=\hKT^{\spa}(\vp).
\]
\end{remark}

\begin{lemma}[cf. {\cite[Theorem 4.2]{KT}} and {\cite[Remark 3.5]{AL3}}]\label{l:3}
Let $\vp:X\multimap X$ be a multivalued map on a compact metric space $X$ and assume that $\vp$ possesses a single-valued continuous selection $f\subset\vp$, i.e. $f(x)\in\vp(x)$, for every $x\in X$. Then the inequality 
\begin{align}\label{eq:A2}
h(f):= \hKT(f)\leq\hKT(\vp)
\end{align}
holds for the topological entropy $\hKT$ in the sense of Definition~\ref{d:7} (see also Remark~\ref{r:2}).
\end{lemma}

\begin{lemma}\label{l:4}
The inequality
\begin{align}\label{eq:A3}
h_+(\vp) \leq \hKT(\vp)
\end{align}
holds for any multivalued l.s.c. map $\vp:X\to\K(X)$ on a compact metric space $X$.
\end{lemma}
\begin{proof}[Sketch of the proof.]
The inequality 
\[
h_+(\vp) \leq h_{\AK}(\vp)
\]
was verified in \cite[Theorem 3]{AL1} for continuous multivalued maps $\vp$, where $h_{\AK}$ stands for the Adler-Konheim-McAndrew type entropy defined in \cite{AK,EK} for u.s.c. maps. The continuity was therefore assumed as an intersection of the u.s.c. requirement in \cite{AK,EK} and the l.s.c. requirement in \cite{AL1}.

Since we have shown in \cite[Theorem 3.4]{AL3} that the upper semicontinuity  of $\vp$ in $h_{\AK}(\vp)$ is superfluous and, especially, that the equality $\hKT(\vp)=h_{\AK}(\vp)$ holds for arbitrary multivalued maps, we can restrict ourselves just to l.s.c. maps.
\end{proof}

\begin{definition}
Let $(X,d)$ be a compact metric space and $\vp:X\multimap X$ be a  multivalued map. For $n\in\N$, we define the pseudometric
\begin{align*}
d_n^{\CM}(x,y) := \inf_{\bar{x}\in\orbn(\vp,x),\bar{y}\in\orbn(\vp,y)} \max_{0\leq i \leq n-1} p(\bar{x}_i,\bar{y}_i), \quad x,y\in X,
\end{align*}
where $\bar{x}=(\bar{x}_0,\ldots,\bar{x}_{n-1})$, $\bar{y}=(\bar{y}_0,\ldots,\bar{y}_{n-1})$ and $x=\bar{x}_0$, $y=\bar{y}_0$.

We call $S\subset X$ an \emph{$(\ep,n)_{\CM}$-separated set for $\vp$} if it is a $(d_n^{\CM},\ep)$-separated subset, i.e. if $d^{\CM}_n(x,y)>\ep$ holds for every pair of distinct points $x,y\in S$. 

We call $R\subset X$ an \emph{$(\ep,n)_{\CM}$-spanning set for $\vp$} if it is a $(d_n^{\CM},\ep)$-spanning subset in $X$, i.e. if for every $y\in X$ there is $x\in R$ such that $d^{\CM}_n(x,y)\leq\ep$.
\end{definition}

\begin{definition}[cf. \cite{CMM1}]\label{d:9}
Let $(X,d)$ be a compact metric space and $\vp: X\multimap X$ be a multivalued map. Denoting by $s_{\CM}(\vp,\ep,n)$ the largest cardinality of an $(\ep,n)_{\CM}$-separated set for $\vp$, we take
\begin{equation*}
\hCM^{\sepp}(\vp,\ep) := \limsup_{n\to\infty} \frac1{n}\log s_{\CM}(\vp,\ep,n).
\end{equation*}
The \emph{topological entropy $\hCM^{\sepp}(\vp)$ of $\vp$} is defined to be
\begin{equation*}
\hCM^{\sepp}(\vp):= \sup_{\ep>0} \hCM^{\sepp}(\vp,\ep).
\end{equation*}

Denoting by $r_{\CM}(\vp,\ep,n)$ the smallest cardinality of an $(\ep,n)_{\CM}$-spanning set for $\vp$, we take
\begin{equation*}
\hCM^{\spa}(\vp,\ep) := \limsup_{n\to\infty} \frac1{n}\log r_{\CM}(\vp,\ep,n).
\end{equation*}
The \emph{topological entropy $\hCM^{\spa}(\vp)$ of $\vp$} is defined to be
\begin{equation*}
\hCM^{\spa}(\vp):= \sup_{\ep>0} \hCM^{\spa}(\vp,\ep).
\end{equation*}
\end{definition}

\begin{lemma}[cf. {\cite[Theorem 3.5]{CMM1}}]\label{l:5}
The inequality
\begin{align}\label{eq:A4}
\hCM^{\spa}(\vp)\leq \hCM^{\sepp}(\vp)
\end{align}
holds for any multivalued map $\vp:X\multimap X$ on a compact metric space $X$.
\end{lemma}

\begin{lemma}[cf. {\cite[Theorem 3.5]{CMM1}}]\label{l:6}
Let $\vp:X\multimap X$ be a multivalued map on a compact metric space $X$ and assume that $\vp$ possesses a single-valued continuous selection $f\subset\vp$, i.e. $f(x)\in\vp(x)$, for every $x\in X$. Then the inequalities
\begin{align}\label{eq:A5}
\hCM^{\sepp}(\vp)\leq h(f):=\hCM^{\sepp}(f) \mbox{ and } \hCM^{\spa}(\vp)\leq h(f):=\hCM^{\spa}(f)
\end{align}
hold for the topological entropies $\hCM^{\sepp}$ and $\hCM^{\spa}$ in the sense of Definition~\ref{d:9}.
\end{lemma}

\section{Theorems of Vivas and Sirvent and their critical analysis}

The first theorem, called in \cite{VS} a ``half variational principle'' reads as follows.

\begin{thm}[cf. {\cite[Theorem 3.2]{VS}}]\label{t:1}
Let $\vp:X\to\K(X)$ be a multivalued map on a compact metric space $X$. Then the inequality
\begin{align}\label{eq:A6}
\hKT(\vp)\geq\sup_{\mu\in\M_1(\vp)} h_\mu(\vp)
\end{align}
holds for the topological entropies $h_\mu(\vp)$ and $\hKT(\vp)$ of $\vp$ in the sense of Definitions~\ref{d:3} and~\ref{d:7} (see Remark~\ref{r:2}), where $\M_1(\vp)$ stands for the set of $\vp$-invariant measures.
\end{thm}

Our critical analysis can be expressed, besides other things, in the following two points.

\textbf{(i)} We will show by the counter-example that the central inequality at condition (10) in its proof in \cite[Theorem 3.2]{VS} is false, namely that the inequality
\begin{align}\label{eq:A7}
\log \card(\P'_n)\leq \log \card(\beta'_n)
\end{align}
does not hold, where $\P'_n$ was defined above and $\beta'_n$ can be defined quite analogously by means of $\beta=\{D_0,D_1,\ldots,D_k\}$, where $D_0=X\setminus \bigcup_{i=1}^{k} D_i$ and $D_i\subset P_i$, for $i=1,\ldots,k$, are compact sets such that $D_i\cap D_j = \emptyset$, for $i\neq j$. The proof presented in \cite{VS} tries to emulate the standard proof of variational principle for single-valued maps (see e.g. \cite[Theorem 8.6]{Wa}), where the compact sets $D_i$ are chosen to be close to $A_i$ (with respect to measure $\mu$), for $i=1,\ldots,k$. The problem is that the compact sets $D_i$ in \cite{VS} are just arbitrary subsets of $A_i$, $i=1,\ldots,k$, without the required closedness.

\begin{example}
Consider the (continuous) tent map $f:[0,1]\to[0,1]$, where
\begin{align*}
f(x):=
\begin{cases}
&x + \frac14,\mbox{ for $x\in[0,\frac12]$,}\\
&\frac54-x, \mbox{ for $x\in(\frac12,1]$.}
\end{cases}
\end{align*}
Taking
\begin{align*}
\P &= \tilde{\P}_0 = \{P_1,P_2\} = \{[0,\frac12],(\frac12,1]\},\\
\beta &= \tilde{\beta}_0 = \{D_0,D_1,D_2\} = \{(0,1),\{0\},\{1\}\},
\end{align*}
(observe that $D_1\subset P_1$, $D_2\subset P_2$ and $D_1\cap D_2 = \emptyset$), we get
\begin{align*}
\tilde{\P}_1 &= \{[0,\frac14]\cup[\frac34,1],(\frac14,\frac34)\},\\
\P'_2  &= \tilde{\P}_0\vee\tilde{\P}_1 = \{(0,\frac14],(\frac14,\frac12],(\frac12,\frac34),[\frac34,1]\},
\end{align*}
where $\card(\P'_2) = 4$.

On the other hand,
\begin{align*}
\tilde{\beta}_1 &= \{[0,1]\},\\
\beta'_2 &= \tilde{\beta}_0\vee\tilde{\beta}_1 = \tilde{\beta}_0 = \{(0,1),\{0\},\{1\}\},
\end{align*}
where $\card(\beta_2')=3$.

Hence, $\card(\P'_2) > \card(\beta'_2)$, which disproves \eqref{eq:A7}, as claimed.
\end{example}

In order to avoid the wrong condition \eqref{eq:A7}, we will proceed in an alternative way under an additional restriction imposed on a multivalued l.s.c. map, see \eqref{eq:A8}. Let us note that the same restrictive conditions have been used in another context in \cite[Proposition 1]{VS}.

\medskip

\textbf{(ii) } Another gap in the proof of \cite[Theorem 3.2]{VS} can be detected in the following argument preceding deduction of \cite[inequality (11)]{VS}: \emph{Let $\NN(\U_n)$ be the minimal cardinality of a subcover of $\U_n$. Since each element of $\U_n$ is the union of at most $2^n$ elements of $\beta_n$ it follows that $\card\beta_n\leq 2^n \NN(\U_n)$, so that $\log \card\beta_n \leq n\log 2 + \log \NN(\U_n)$.}

Since the mapping under consideration is a general multivalued map with closed values, $\U_n$ are not necessarily open covers (see also Remark~\ref{r:9} below), and $\NN(\U_n)$ may be an infinite cardinal. Even if we assume its finiteness, one cannot agree with the quoted reasoning, because the elements of $\beta_n$ need not be pairwise disjoint (in contrast to the single-valued case). Therefore, the cardinality of $\beta_n$ cannot be bounded from above by $2^n\NN(\U_n)$.

In the sequel, we will avoid the usage of the wrong argument by the application of Lemma~\ref{l:1}. 

\bigskip

Although the proof of the following theorem employs the arguments of an elegant proof of single valued variational principle due to Misiuriewicz (see \cite{Mi}), it must have been furnished with some new ideas.

\begin{thm}\label{t:2}
Let $\vp:X\multimap X$ be an l.s.c. multivalued map on a compact metric space $X$ and let $\mu\in\M_1(\vp)$ satisfy
\begin{align}\label{eq:A8}
\mu(\vp_{+}^{-1}(A))=\mu(A) \mbox{ and } \mu(\vp_{+}^{-1}(A)\cap\vp_{+}^{-1}(B)) = \mu(\vp_{+}^{-1}(A\cap B)),
\end{align}
for all $A,B\in\B$. Then
\begin{align}\label{eq:hvp}
h_{\mu}(\vp) \leq h_+(\vp)
\end{align}
holds for the topological entropies $h_\mu(\vp)$ and $h_+(\vp)$ of $\vp$ in the sense of Definitions~\ref{d:3} and~\ref{d:4}.
\end{thm}

\begin{proof}

Let $\P=\{P_1,P_2,\ldots, P_k\}$ be a finite ordered measurable partition of $X$. Choose $\ep>0$ so that $\ep\leq \frac1{k \log k}$. Since $\mu$ is regular, there exist compact sets $D_i\subset P_i$, $i=1,\ldots,k$, such that $\mu(P_i\setminus D_i)<\ep$. Let $\beta$ be the partition $\beta=\{D_0,D_1,\ldots,D_k\}$, where $D_0= X\setminus\bigcup_{i=1}^{k} D_i$.

Clearly, $\mu(D_0) \leq k\ep$, and the conditional entropy of a partition $\P$ with respect to $\beta$ satisfies
\begin{align}
H_{\mu}(\P\vert\beta) &= -\sum_{i=0}^{k}\sum_{j=1}^{k} \mu(D_i)\Phi\left(\frac{\mu(D_i\cap P_j)}{\mu(D_i)}\right)\nonumber \\ 
&= - \mu(D_0)\sum_{j=1}^{k} \Phi\left( \frac{\mu(D_0\cap P_j)}{\mu(D_0)}\right)  \label{eq:2} \\ 
&\leq \mu(B_0)\log k < k\ep \log k < 1. \label{eq:3}
\end{align}

The equality \eqref{eq:2} is satisfied, because if $i\neq 0$, then $\frac{\mu(P_i\cap D_j)}{\mu(P_i)} \in\{0,1\}$ and $\Phi(0)=\Phi(1)=0$. The inequality \eqref{eq:3} follows from the convexity of $\Phi$.

We define an open cover $\U=\{D_0\cup D_1,\ldots,D_0\cup D_k\}$. Now, for $n\geq 1$, according to Lemma~\ref{l:1},
\[
H_{\mu}(\bigvee_{i=0}^{n-1} \tilde{\beta_i}) \leq \log \NZ(\bigvee_{i=0}^{n-1} \tilde{\beta_i}).
\]

If $U\in\bigvee_{i=0}^{n-1} \vp^{-i}_{+}(\U)$, then (see e.g. \cite[Proposition 1.2.2]{HP}),
\begin{align*}
U = \bigcap_{i=0}^{n-1} \vp_{+}^{-i}(D_0\cup D_{j_i}) = \bigcap_{i=0}^{n-1} \left(\vp_{+}^{-i}(D_0)\cup\vp_{+}^{-i}(D_{j_i})\right),
\end{align*}
where $j_i\in \{1,\ldots,k\}$, $i\in\{0,\ldots,n-1\}$.

It follows that each element of $\bigvee_{i=0}^{n-1}\vp^{-i}_{+}(\U)$ can be written as a union of $2^n$ sets of $\bigvee_{i=0}^{n-1}\vp_{+}^{-i}(\beta)$. Each set of $\bigvee_{i=0}^{n-1}\vp_{+}(\beta)$ contains a set of $\bigvee_{i=0}^{n-1} \tilde{\beta_i}$ and the difference of these two corresponding sets has $\mu$-measure zero. Justification of this observation requires a simple and straightforward computation which has been done in the proof of \cite[Proposition 1]{VS}. 

Moreover, the sets of $\bigvee_{i=0}^{n-1} \tilde{\beta_i}$ are pairwise disjoint. We prove that
\[
\NZ(\bigvee_{i=0}^{n-1}\tilde{\beta}_i)\leq 2^n \NN(\bigvee_{i=0}^{n-1}\vp^{-i}_{+}(\U)).
\]
Assume the contrary. Then there is $U\in\bigvee_{i=0}^{n-1}\vp^{-i}_{+}(\U)$ containing more than $2^n$ elements of $\bigvee_{i=0}^{n-1}\tilde{\beta}_i$ with a positive measure. This is impossible, because $U$ is a union of at most $2^n$ elements of $\bigvee_{i=0}^{n-1}\tilde{\beta}_i$ with a positive $\mu$-measure and a set of $\mu$-measure zero.

Therefore,
\begin{align}
H_{\mu}(\bigvee_{i=0}^{n-1}\tilde{\beta}_i)&\leq \log  2^n \NN(\bigvee_{i=0}^{n-1}\vp^{-i}_{+}(\U)) = \log(\bigvee_{i=0}^{n-1}\vp^{-i}_{+}(\U)) + n\log 2,\nonumber\\
h_{\mu}(\vp,\P)  &\leq h_{\mu}(\vp,\beta) + H_{\mu}(\P\vert\beta) = \limsup_{n\to\infty} \frac1{n}H_{\mu}(\bigvee_{i=0}^{n-1}\tilde{\beta}_i) + H_{\mu}(\P\vert\beta) \label{eq:con}\\
&\leq \limsup_{n\to\infty} \frac1{n} \log(\bigvee_{i=0}^{n-1}\vp^{-i}_{+}(\U)) +\log 2 + 1 = h_{+}(\vp,\U) + \log 2 + 1\nonumber\\
&\leq h_{+}(\vp) + \log 2 + 1 \nonumber.
\end{align}

The crucial inequality \eqref{eq:con} was proved as in \cite[condition (6)]{VS}, under the assumptions \eqref{eq:A8}.

Finally, taking the supremum over all finite ordered measurable partitions $\P$ we obtain
\begin{align}\label{eq:semi-var-0}
h_{\mu}(\vp) \leq h_{+}(\vp) + \log 2 + 1.
\end{align}

We intend to apply \eqref{eq:semi-var-0} for $\vp^n:X\multimap X$, where $n\in\N$. We need to verify the assumptions. Firstly, the multivalued map $\vp^n$ is l.s.c., for all $n\in\N$, by \cite[Proposition 1.2.56]{HP}. Secondly, the condition \eqref{eq:A8} for $\vp^n$, $n\in\N$, can be proved by the mathematical induction. Thus, $\mu\in\M_1(\vp^n)$ and
\[
h_{\mu}(\vp^n) \leq h_{+}(\vp^n) + \log 2 + 1.
\]

By \cite[Theorem 3.1]{VS} and Lemma~\ref{l:iter}, we obtain
\begin{align*}
n \cdot h_{\mu}(\vp) &\leq h_{\mu}(\vp^n) \leq h_{+}(\vp^n) + \log 2 + 1 \leq n \cdot h_{+}(\vp) + \log 2 + 1,\\
h_{\mu}(\vp) &\leq h_{+}(\vp) + \frac{\log 2 + 1}{n}.
\end{align*}

For $n$ tending to infinity, we arrive at the desired inequality \eqref{eq:hvp}.
\end{proof}

\begin{remark}
Since $\mu(\vp^{-1}_+(A))\geq\mu(A)$ holds by Definition~\ref{d:2}, $\vp^{-1}_+(A)\cap\vp^{-1}_+(B)\supset\vp^{-1}_+(A\cap B)$ in general by e.g. \cite[Proposition 1.2.2]{HP}, and subsequently
\[
\mu(\vp^{-1}_+(A)\cap\vp^{-1}_+(B))\geq \mu(\vp^{-1}_+(A\cap B))
\]
we can rewrite \eqref{eq:A8} without any loss of generality into
\[
\mu(\vp_{+}^{-1}(A))\leq\mu(A) \mbox{ and } \mu(\vp_{+}^{-1}(A)\cap\vp_{+}^{-1}(B)) \leq \mu(\vp_{+}^{-1}(A\cap B)),
\]
for all $A,B\in\B$.
\end{remark}

The inequality \eqref{eq:A6} in Theorem~\ref{t:1} can be then easily improved, under the additional conditions in Theorem~\ref{t:2}, in the following way.

\begin{corollary}\label{c:1}
Under the assumptions of Theorem~\ref{t:2}, where $\vp: X\to\KK(X)$, the inequalities
\[
\hKT(\vp) \geq h_+(\vp) \geq h_{\mu}(\vp)
\]
hold, where $\hKT(\vp)$ denotes the topological entropy of $\vp$ in the sense of Definition~\ref{d:7} (see Remark~\ref{r:2}).
\end{corollary}
\begin{proof}
The conclusion follows directly from Theorem~\ref{t:2}, by means of \eqref{eq:A3} in Lemma~\ref{l:4}.
\end{proof}

Another extension of Theorem~\ref{t:2} reads as follows.
\begin{corollary}\label{c:2}
Under the assumptions of Theorem~\ref{t:2}, where $\vp:X\to\KK(X)$ is additionally continuous, the relations
\[
\sup_{\nu\in\M_1(\vp^*)}h_\nu(\vp^*) = h(\vp^*) \geq h_+(\vp) \geq h_\mu(\vp)
\]
are satisfied for the topological entropy $h(\vp^*):=h_+(\vp^*)$ and the metric entropy $h_\nu(\vp^*)$, $\nu\in\M_1(\vp^*)$, of the single-valued continuous hypermap $\vp^*:\KK(X)\to\KK(X)$, where $\vp^*(A):=\bigcup_{x\in A} \vp(x)$, on a compact hyperspace $\KK(X)$ endowed with the Hausdorff metric $\dH$, i.e.
\[
\dH(A,B) := \max \{\sup_{a\in A}(\inf_{b\in B} d(a,b), \sup_{b\in B}(\inf_{a\in A} d(a,b)))\},
\]
for $A,B\in\KK(X)$.
\end{corollary}
\begin{proof}
The conclusion follows directly from Theorem~\ref{t:2}, by means of \cite[Theorem 7]{AL1} and the standard variational principle for single-valued continuous maps in \cite{Go3,Mi,Wa}.
\end{proof}

\begin{remark}\label{r:4}
Because of the guaranteed existence of an invariant measure $\mu\in\M_1(\vp)$ discussed in Section~\ref{sec:2}, the multivalued maps $\vp:X\to\KK(X)$ in Theorem~\ref{t:2} and its Corollaries~\ref{c:1}, \ref{c:2} need not have convex values and need not be considered in a Banach space, provided they are, for instance, continuous on a compact metric space $X$. Let us note that, in Theorem~\ref{t:1}, the authors have assumed neither the upper semicontinuity of maps $\vp$ nor the convexity of their values for the existence of invariant measures. In fact, they defined incorrectly the closed maps as closed-valued maps, but not as those having a closed graph on a compact set, i.e. not as u.s.c. maps. Of course, both above conditions are sufficient, but not necessary. In other words, the existence of invariant measures $\mu\in\M_1(\vp)$ in Theorem~\ref{t:2} and its Corollaries~\ref{c:1}, \ref{c:2} can be supposed explicitly just for l.s.c. maps $\vp:X\to\KK(X)$ on a compact metric space $X$ or one can assume that $\vp$ possesses a single-valued continuous selection $f\subset\vp$.
\end{remark}

The second problematic theorem concerns the lower estimate of metric entropy.

\begin{thm}[cf. {\cite[Theorem 3.7]{VS}}]\label{t:3}
Let $\vp:X\to\KK(X)$ be an l.s.c. multivalued map on a compact metric space $X$. If $h_\mu(\vp)\geq h_\mu(f)$ holds for every single-valued continuous selection $f\subset\vp$ of $\vp$ and every $f$-invariant measure $\mu$, then the inequality
\begin{align}\label{eq:A11}
\sup_{\mu\in\M_1(\vp)} h_\mu(\vp) \geq \hCM^{\sepp}(\vp)
\end{align}
is satisfied for the entropies $h_\mu(\vp)$ and $\hCM^{\sepp}(\vp)$ of $\vp$ in the sense of Definitions~\ref{d:3} and~\ref{d:9}, where $\M_1(\vp)$ stands for the set of $\vp$-invariant measures.
\end{thm}

The proof of Theorem~\ref{t:3} was based, besides other things, on the following lemma.

\begin{lemma}[cf. {\cite[Lemma 3.6]{VS}}]\label{l:7}
Every l.s.c. multivalued map $\vp:X\to\KK(X)$ on a compact metric space $X$ admits a single-valued continuous selection $f\subset\vp$ of $\vp$, i.e. $f(x)\in\vp(x)$, for every $x\in X$.
\end{lemma}

Lemma~\ref{l:7} is evidently false (even for continuous multivalued maps), as documented by the following example.

\begin{example}[cf. {\cite[Example 2.10]{ATDZ}} and {\cite[Example 1.4.4]{HP}}]
Let  $D = \{x\in\R^2: \|x\|^2 = x_1^2+x_2^2 \leq 1\}$ be the closed unit disc in $\R^2$ and $S^1 = \partial D$ its boundary. Let $\vp:D\to\KK(\R^2)$be defined as
\[
\vp(x)=
\begin{cases}
&S^1\setminus\{y\in\R^2: \|y-x\|y\|^{-1}\| <\|x\|\}, \mbox{ for $x\neq 0$,}\\
&S^1, \mbox{ for $x=0$.}
\end{cases}
\]

Then $\vp$ is l.s.c. (even continuous) on $D$, but does not admit any continuous selection, because if it exists, then $\hat{x}$ would be its fixed point (by the Brower fixed point theorem) such that  $\hat{x} = f(\hat{x})\in \vp(\hat{x}) \subset S^1$ and $\hat{x}\neq 0$. Hence, $\|\hat{x} - \hat{x}\|\hat{x}\|^{-1}\|\geq \|\hat{x}\| = 1$, a contradiction.
\end{example}

As pointed out in \cite[p. 14]{ATDZ}, apart from the required convexity, the assertion of the Michael selection theorem (see Proposition~\ref{p:1}) becomes false without the closedness assumption on the images $\vp(x)$, or without the completeness assumption on the normed space $Y$ in Proposition~\ref{p:1}. On the other hand, in some special Banach spaces $Y$, the convexity of values of $\vp(x)$ can be replaced by their decomposability in the sense of Rockafellar (see e.g. \cite{Fr}). Therefore, for another sufficient condition for the existence of a single-valued continuous selection $f\subset\vp$ of an l.s.c. map $\vp:X\to\KK(X)$ on a compact subset of the same special Banach spaces, it is enough that $\vp$ has decomposable values.

Correcting the proof of Theorem~\ref{t:3} by means of Proposition~\ref{p:1}, instead of the wrong Lemma~\ref{l:7}, the obtained statement can be formulated in the form of the following proposition.

\begin{prop}\label{p:2}
Let $\vp:X\to\KK(X)$ be a convex-valued l.s.c. multivalued map on a compact subset $X$ of a Banach space. If $h_\mu(\vp)\geq h_\mu(f)$ is satisfied for every single-valued continuous selection $f\subset\vp$ of $\vp$ and every $f$-invariant measure $\mu$, then (cf. \eqref{eq:A11})
\begin{align*}
\sup_{\mu\in\M_1(\vp)} h_\mu(\vp) \geq \hCM^{\sepp}(\vp)\geq \hCM^{\spa}(\vp),
\end{align*}
and
\begin{align*}
\sup_{\mu\in\M_1(\vp)} h_\mu(\vp) \geq h_+(\vp)
\end{align*}
hold for the entropies $h_\mu(\vp)$, $h_+(\vp)$, $\hCM^{\sepp}(\vp)$ and $\hCM^{\spa}(\vp)$ of $\vp$ in the sense of Definitions~\ref{d:3}, \ref{d:4} and~\ref{d:9}, respectively.
\end{prop}
\begin{proof}
By the standard variational principle for single-valued continuous maps (see e.g. \cite{Mi}, \cite{Wa}), we have
\[
\sup_{\mu\in\M_1(f)} h_\mu(f) = h(f),
\]
for every single-valued continuous selection $f\subset \vp$ of $\vp$, whose existence if guaranteed by Proposition~\ref{p:1}.

Furthermore, in view of \eqref{eq:A4} in Lemma~\ref{l:5}, \eqref{eq:A5} in Lemma~\ref{l:6}, and \eqref{eq:A1} in Lemma~\ref{l:2},
\[
h(f) \geq \hCM^{\sepp}(\vp)\geq\hCM^{\spa}(\vp),
\]
and 
\[
h(f) \geq h_+(\vp).
\]

Thus, by virtue of the hypothesis 
\[
h_\mu(\vp)\geq h_\mu(f),
\]
we arrive in the end at
\begin{align*}
\sup_{\mu\in\M_1(\vp)} h_\mu(\vp) &\geq \sup_{\mu\in\M_1(f)} h_\mu(f) = h(f) \geq \hCM^{\sepp}(\vp)\geq \hCM^{\spa}(\vp),\\
\sup_{\mu\in\M_1(\vp)} h_\mu(\vp) &\geq \sup_{\mu\in\M_1(f)} h_\mu(f) = h(f) \geq h_+(\vp),
\end{align*}
which completes the proof.
\end{proof}

Replacing the assumed inequality by the reverse one, we can also give by analogous arguments the following proposition, related rather to Theorem~\ref{t:1} (cf. \eqref{eq:A6}) than Theorem~\ref{t:3}.

\begin{prop}\label{p:3}
Let $\vp:X\to\KK(X)$ be a convex-valued l.s.c. multivalued map on a compact subset $X$ of a Banach space. If $h_\mu(\vp)\leq h_\mu(f)$ is satisfied for every single-valued continuous selection $f\subset\vp$ of $\vp$ and every $f$-invariant measure $\mu$, then (cf. \eqref{eq:A11})
\begin{align*}
\sup_{\mu\in\M_1(\vp)} h_\mu(\vp) \leq \hKT(\vp),		
\end{align*}
holds for the entropies $h_\mu(\vp)$ and $\hKT(\vp)$ of $\vp$ in the sense of Definitions~\ref{d:3} and~\ref{d:7} (see Remark~\ref{r:2}).
\end{prop}

\begin{proof}
By the well known inequality for single-valued continuous maps (see e.g. \cite{Go1}), we have
\[
\sup_{\mu\in\M_1(f)} h_\mu(f) \leq h(f),
\]
for every single-valued continuous selection $f\subset \vp$ of $\vp$, whose existence if guaranteed by Proposition~\ref{p:1}.

Furthermore, in view of \eqref{eq:A2} in Lemma~\ref{l:3},
\[
h(f) \leq \hKT(\vp).
\]

Thus, by virtue of the hypothesis
\[
h_\mu(\vp) \leq h_\mu(f),
\]
we arrive finally at
\[
\sup_{\mu\in\M_1(\vp)} h_\mu(\vp) \leq \sup_{\mu\in\M_1(f)} h_\mu(f) \leq h(f) \leq \hKT(\vp),		
\]
which completes the proof.

\end{proof}

\begin{remark}\label{r:5}
Since the supposed inequalities in Propositions~\ref{p:2} and~\ref{p:3} are quite implicit and non-effective, we decided to fomulate the obtained results rather in the form of propositions than theorems as in \cite{VS}.
\end{remark}

\section{Full-variational principle}
In spite of the restrictions mentioned in Remark~\ref{r:5}, we will be finally able to establish a full variational principle for special multivalued l.s.c. maps. Hence, summing up the results of Theorem~\ref{t:2} and Proposition~\ref{p:2}, we can give immediately the last theorem.

\begin{thm}[Full variational principle]\label{t:4}
Let $\vp:X\to\KK(X)$ be a convex-valued l.s.c. multivalued map on a compact subset $X$ of a Banach space satisfying conditions \eqref{eq:A8}, for every $\mu\in\M_1(\vp)$ and all $A,B\in\B$. If $h_\mu(\vp)\geq h_\mu(f)$ is still satisfied for every single-valued continuous selection $f\subset\vp$ of $\vp$ and every $f$-invariant measure $\mu$, then the equality
\begin{align*}
\sup_{\mu\in\M_1(\vp)} h_\mu(\vp) = h_+(\vp)
\end{align*}
holds for the entropies $h_\mu(\vp)$ and $h_+(\vp)$ of $\vp$ in the sense of Definitions~\ref{d:3} and \ref{d:4}.
\end{thm}

\begin{remark}\label{r:6}
In view of Remark~\ref{r:4}, the l.s.c. maps $\vp:X\to\KK(X)$ in Theorem~\ref{t:4} can be more generally replaced there by their union with single-valued continuous maps $f_i: X\to X$, $i=1,\ldots,n$, i.e. by $\vp\cup (\bigcup_{i=1}^{n}f_i)$.
\end{remark}

\begin{remark}\label{r:7}
Theorem~\ref{t:4} is nonempty, because it obviously holds for single-valued continuous functions. It is however a question how ``rich'' is the class of multivalued maps satisfying this variational principle. In any way, in view of Proposition~\ref{p:2} and Lemma~\ref{l:4}, it can be extended into the following relations
\begin{align*}
\hKT(\vp) \geq h_+(\vp) = \sup_{\mu\in\M_1(\vp)} h_\mu(\vp) \geq \hCM^{\sepp}(\vp) \geq \hCM^{\spa}(\vp).
\end{align*}

If $\vp$ is additionally continuous then, according to Corollary~\ref{c:2}, we have also
\begin{align*}
\sup_{\nu\in\M_1(\vp^*)} h_\nu(\vp^*) = h(\vp^*) \geq h_+(\vp) = \sup_{\mu\in\M_1(\vp)} h_\mu(\vp) \geq \hCM^{\sepp}(\vp) \geq \hCM^{\spa}(\vp).
\end{align*}
\end{remark}

Since the conditions \eqref{eq:A8} seem to be extremely restrictive, we will conclude by the trivial illustrative example, where they are satisfied for a u.s.c. mapping.

\begin{example}\label{e:4}
Consider the u.s.c. map $\vp:[0,1]\to\KK([0,1])$, where
\begin{align*}
	\vp(x) =
	\begin{cases}
		&\{0,1\}, \mbox{ for $x\in\{0,1\}$,}\\
		&x,\mbox{for $x\in(0,1)$.}
	\end{cases}
\end{align*}

One can easily check that the strong invariance condition $\mu(\vp^{-1}_+(A)) = \vp(A)$ is satisfied, jointly with $\mu(\vp^{-1}_+(A\cap B))=\mu(A\cap B)$, i.e. conditions \eqref{eq:A8}, for the standard Lebesgue measure $\mu$ on $[0,1]$ and all the Lebesgue measurable subsets $A,B\subset [0,1]$.

On the other hand, the mapping $\vp$ under consideration is here neither l.s.c., nor continuous, but u.s.c., which is not the case at Theorems~\ref{t:2} and~\ref{t:4} (cf. also Remark~\ref{r:6}). The same is true for another u.s.c. mapping, satisfying conditions \eqref{eq:A8}, which was constructed in \cite[Example 2.5]{VS}.
\end{example}

\begin{remark}\label{r:9}
In fact, it is also not the case at Theorem~\ref{t:1}, despite the fact that no regularity restrictions have been explicitly imposed there on $\vp:X\to\KK(X)$ (cf. Remark~\ref{r:4}). Even if the authors of Theorem~\ref{t:1} would have misunderstood by closed-valued maps the u.s.c. ones, they required in the proof of Theorem~\ref{t:1} an openness of the covers $\U_n$, $n\in\N$, where $\U=\{D_0\cup D_1,\ldots,D_0\cup D_k\}$ and
\[
\U_n = \left\{ \bigcap_{i=0}^{n-1} \vp^{-i}_+(U_i): U_i\in\U, i=0,\ldots,n-1\right\}.
\]
This is possible only when $\vp$ is l.s.c. (see Definition~\ref{d:1}). Therefore, their maps $\vp$ are in Theorem~\ref{t:1} de facto continuous.
\end{remark}

\bibliography{AL-Vivas-Sirvent}\bibliographystyle{siam}

\begin{thebibliography}{10}

\bibitem{AKM}
{\sc R.~L. Adler, A.~G. Konheim, and M.~H. McAndrew}, {\em Topological
  entropy}, Trans. Amer. Math. Soc., 114 (1965), pp.~309--319.

\bibitem{AK}
{\sc L.~Alvin and J.~P. Kelly}, {\em Topological entropy of {M}arkov set-valued
  functions}, Ergodic Theory Dynam. Systems, 41 (2021), pp.~321--337.

\bibitem{An}
{\sc J.~Andres}, {\em Chaos for multivalued maps and induced hyperspace maps},
  Chaos, Solitons \& Fractals, 138 (2020), pp.~109898, 8.

\bibitem{AF}
{\sc J.~Andres and J.~Fi\v{s}er}, {\em Metric and topological multivalued
  fractals}, Internat. J. Bifur. Chaos Appl. Sci. Engrg., 14 (2004),
  pp.~1277--1289.

\bibitem{AG}
{\sc J.~Andres and L.~G\'{o}rniewicz}, {\em Topological fixed point principles
  for boundary value problems}, vol.~1 of Topological Fixed Point Theory and
  Its Applications, Kluwer Academic Publishers, Dordrecht, 2003.

\bibitem{AL1}
{\sc J.~Andres and P.~Ludv\'{\i}k}, {\em Topological entropy of multivalued
  maps in topological spaces and hyperspaces}, Chaos, Solitons \& Fractals, 160
  (2022), pp.~Paper No. 112287, 11.

\bibitem{AL2}
\leavevmode\vrule height 2pt depth -1.6pt width 23pt, {\em Parametric
  topological entropy of families of multivalued maps in topological spaces and
  induced hyperspace maps}, Commun. Nonlinear Sci. Numer. Simul., 125 (2023),
  pp.~Paper No. 107395, 16.

\bibitem{AL3}
\leavevmode\vrule height 2pt depth -1.6pt width 23pt, {\em Parametric
  topological entropy on orbits of arbitrary multivalued maps in compact
  {H}ausdorff spaces}, J. Math. Anal. Appl., 540 (2024), p.~Paper No. 128588.

\bibitem{ATDZ}
{\sc J.~Appell, E.~De~Pascale, N.~H. Th\'{a}i, and P.~P. Zabre\u{\i}ko}, {\em
  Multi-valued superpositions}, Dissertationes Math. (Rozprawy Mat.), 345
  (1995), p.~97.

\bibitem{Ba1}
{\sc F.~Balibrea}, {\em On the origin and development of some notions of
  entropy}, Topol. Algebra Appl., 3 (2015), pp.~57--74.

\bibitem{Ba}
\leavevmode\vrule height 2pt depth -1.6pt width 23pt, {\em On problems of
  topological dynamics in non-autonomous discrete systems}, Appl. Math.
  Nonlinear Sci., 1 (2016), pp.~391--403.

\bibitem{Ba2}
{\sc F.~Balibrea}, {\em On notions of entropy in physics and mathematics}, in
  Research trends and challenges in physical science, T.~George, ed., vol.~7,
  B.P. International, 2022, pp.~92--111.

\bibitem{Bo}
{\sc R.~Bowen}, {\em Entropy for group endomorphisms and homogeneous spaces},
  Trans. Amer. Math. Soc., 153 (1971), pp.~401--414.

\bibitem{CMM2}
{\sc D.~Carrasco-Olivera, R.~Metzger, and C.~A. Morales}, {\em Logarithmic
  expansion, entropy, and dimension for set-valued maps}, in Optimal control
  ({R}ussian), vol.~178 of Itogi Nauki Tekh. Ser. Sovrem. Mat. Prilozh. Temat.
  Obz., Vseross. Inst. Nauchn. i Tekhn. Inform. (VINITI), Moscow, 2020,
  pp.~31--40.
\newblock Translation in J. Math. Sci. {{\bf{2}}76} (2023), no. 2, 227--236.

\bibitem{CMM1}
{\sc D.~Carrasco-Olivera, R.~Metzger~Alvan, and C.~A. Morales~Rojas}, {\em
  Topological entropy for set-valued maps}, Discrete Contin. Dyn. Syst. Ser. B,
  20 (2015), pp.~3461--3474.

\bibitem{CP}
{\sc W.~Cordeiro and M.~J. Pac\'{\i}fico}, {\em Continuum-wise expansiveness
  and specification for set-valued functions and topological entropy}, Proc.
  Amer. Math. Soc., 144 (2016), pp.~4261--4271.

\bibitem{Di}
{\sc E.~I. Dinaburg}, {\em A connection between various entropy
  characterizations of dynamical systems}, Izv. Akad. Nauk SSSR Ser. Mat., 35
  (1971), pp.~324--366.

\bibitem{Do}
{\sc T.~Downarowicz}, {\em Positive topological entropy implies chaos {DC}2},
  Proc. Amer. Math. Soc., 142 (2014), pp.~137--149.

\bibitem{EK}
{\sc G.~Erceg and J.~Kennedy}, {\em Topological entropy on closed sets in
  {$[0,1]^2$}}, Topology Appl., 246 (2018), pp.~106--136.

\bibitem{FN}
{\sc W.~Freeden}, ed., {\em Frontiers in entropy across the
  disciplines---panorama of entropy: theory, computation, and applications},
  vol.~4 of Contemporary Mathematics and Its Applications: Monographs,
  Expositions and Lecture Notes, World Scientific Publishing Co. Pte. Ltd.,
  Hackensack, NJ, [2023] \copyright 2023.

\bibitem{Fr}
{\sc A.~Fryszkowski}, {\em Fixed point theory for decomposable sets}, vol.~2 of
  Topological Fixed Point Theory and Its Applications, Kluwer Academic
  Publishers, Dordrecht, 2004.

\bibitem{Go3}
{\sc T.~N.~T. Goodman}, {\em Relating topological entropy and measure entropy},
  Bull. London Math. Soc., 3 (1971), pp.~176--180.

\bibitem{Go1}
{\sc L.~W. Goodwyn}, {\em Topological entropy bounds measure-theoretic
  entropy}, Proc. Amer. Math. Soc., 23 (1969), pp.~679--688.

\bibitem{Go2}
\leavevmode\vrule height 2pt depth -1.6pt width 23pt, {\em Comparing
  topological entropy with measure-theoretic entropy}, Amer. J. Math., 94
  (1972), pp.~366--388.

\bibitem{HP}
{\sc S.~Hu and N.~S. Papageorgiou}, {\em Handbook of multivalued analysis.
  {V}ol. {I}}, vol.~419 of Mathematics and its Applications, Kluwer Academic
  Publishers, Dordrecht, 1997.
\newblock Theory.

\bibitem{Ka}
{\sc A.~Katok}, {\em Fifty years of entropy in dynamics: 1958--2007}, J. Mod.
  Dyn., 1 (2007), pp.~545--596.

\bibitem{KT}
{\sc J.~P. Kelly and T.~Tennant}, {\em Topological entropy of set-valued
  functions}, Houston J. Math., 43 (2017), pp.~263--282.

\bibitem{LLZ}
{\sc X.~Li, Z.~Li, and Y.~Zhang}, {\em Thermodynamic formalism for
  correspondences}.
\newblock https://doi.org/10.48550/arXiv.2311.09397, 2024.

\bibitem{Ma}
{\sc M.~Malkin}, {\em On continuity of entropy of discontinuous mappings of the
  interval}, Selecta Mathematica Sovietica, 8 (1989), pp.~131--139.

\bibitem{Mi}
{\sc M.~Misiurewicz}, {\em Horseshoes for mappings of the interval}, Bull.
  Acad. Polon. Sci. S\'{e}r. Sci. Math., 27 (1979), pp.~167--169.

\bibitem{RT}
{\sc B.~E. Raines and T.~Tennant}, {\em The specification property on a
  set-valued map and its inverse limit}, Houston J. Math., 44 (2018),
  pp.~665--677.

\bibitem{TWY}
{\sc E.~Tarafdar, P.~Watson, and X.-Z. Yuan}, {\em Poincare's recurrence
  theorems for set-valued dynamical systems}, Appl. Math. Lett., 10 (1997),
  pp.~37--44.

\bibitem{VS}
{\sc K.~J. Vivas and V.~F. Sirvent}, {\em Metric entropy for set-valued maps},
  Discrete Contin. Dyn. Syst. Ser. B, 27 (2022), pp.~6589--6604.

\bibitem{Wa}
{\sc P.~Walters}, {\em An introduction to ergodic theory}, vol.~79 of Graduate
  Texts in Mathematics, Springer-Verlag, New York-Berlin, 1982.

\bibitem{WZZ}
{\sc X.~Wang, Y.~Zhang, and Y.~Zhu}, {\em On various entropies of set-valued
  maps}, J. Math. Anal. Appl., 524 (2023), pp.~Paper No. 127097, 37.

\end{thebibliography}

\end{document}